%% file: Posterior_concentration_rate_manuscript.tex
\documentclass[aos,preprint]{imsart}

\RequirePackage[OT1]{fontenc}
\RequirePackage{amsthm,amsmath,mathrsfs}
\RequirePackage{amssymb}
\RequirePackage{graphicx}
\RequirePackage{bbm}
\RequirePackage[numbers]{natbib}
\RequirePackage{multirow}
\RequirePackage[colorlinks,citecolor=blue,urlcolor=blue]{hyperref}

\startlocaldefs
\numberwithin{equation}{section}
\theoremstyle{plain}

\newtheorem{theorem}{Theorem}[section]
\newtheorem{lemma}{Lemma}[section]

\newtheorem{definition}{Definition}[section]

\theoremstyle{definition}
\newtheorem{remark}{Remark}[section]

\makeatletter

\newcommand{\Rmnum}[1]{\expandafter\@slowromancap\romannumeral #1@}
\makeatother

\endlocaldefs

\begin{document}

\begin{frontmatter}

\title{Multivariate Density Estimation via Adaptive Partitioning (II): Posterior Concentration}
\runtitle{Density Estimation Via Adaptive Partitioning}

\begin{aug}
\author{\fnms{Linxi} \snm{Liu}\thanksref{t2,m1}\ead[label=e1]{linxiliu@stanford.edu}}
\and
\author{\fnms{Wing Hung} \snm{Wong}\thanksref{t2,m1,m2}\ead[label=e2]{whwong@stanford.edu}}

\thankstext{t2}{Supported by NIH grant R01GM109836, and NSF grants DMS1330132 and DMS1407557.}
\runauthor{L. Liu and W. H. Wong}

\affiliation{Department of Statistics, Stanford University\thanksmark{m1}\\Department of Health Research and Policy, Stanford University\thanksmark{m2}}

\address{Department of Statistics\\
Stanford University\\
390 Serra Mall, Sequoia Hall\\
Stanford, California 94305\\
USA\\
}

\end{aug}

\begin{abstract}
In this paper, we study a class of non-parametric density estimators
under Bayesian settings. The
estimators are piecewise constant
functions on binary
partitions. We analyze the concentration rate of the
posterior distribution under a suitable prior, and demonstrate that the
rate does not directly depend on the dimension of the
problem. This paper can be viewed as an extension of
(\cite{LW2014}) where the convergence rate of a related sieve
MLE was established. Compared to the sieve MLE, the main
advantage of the Bayesian method is that it can adapt to
the unknown complexity of the true density function, thus achieving the
optimal convergence rate without artificial conditions on the density. 

\end{abstract}

\begin{keyword}[class=MSC]
\kwd[Primary ]{62G20, }
\kwd[secondary ]{62H10.}
\end{keyword}

\begin{keyword}
\kwd{density estimation}
\kwd{posterior concentration rate}
\kwd{adaptive partitioning.}
\end{keyword}

\end{frontmatter}

\input{intro}

\input{result}
\input{upperbound}
\input{lowerbound}
\input{proof}

\appendix

\section*{Acknowledgements}
The authors would like to thank Bai Jiang for helpful discussions.

\bibliographystyle{imsart-nameyear}
\bibliography{ref} 

\end{document}

%% file: intro.tex
\section{Introduction}
In this paper, we study the asymptotic behavior of posterior
distributions of a class of density
estimators based on adaptive partitioning. Density estimation is a
fundamental problem in statistics---once an explicit estimate
of the density function is obtained, various kinds of
statistical inference can follow, including nonparametric testing, clustering,
and data compression. 

With univariate (or bivariate) data, the most basic non-parametric
method for density estimation is the histogram method.  In this
method, the sample space is partitioned into regular intervals (or
rectangles), and the density is estimated by the relative frequency of data
points falling into each interval (rectangle). However, this method is
of limited utility in higher dimensional spaces because the number of
cells in a regular partition of a $p$-dimensional space will grow
exponentially with $p$, which makes the relative frequency highly
variable unless the sample size is extremely large.  In this situation
the histogram may be improved by adapting the partition to the data so
that larger rectangles are used in the part of the sample space where
data is sparse. Motivated by this consideration, researchers
have recently developed several
multivariate density estimation methods based on
adaptive partitioning. For example, by generalizing the classical
Polya Tree construction (\cite{ferguson1974}), \cite{wong2010}
developed the Optional Polya Tree (OPT) prior on the space of simple
functions. In this prior the partition that supports the simple
function is generated by a random recursive partitioning process. As
the partition is random a priori, it can be inferred from its
posterior distribution once the data is observed. Computational issues
related to OPT density estimates were discussed in \cite{LJW}, where
efficient algorithms were developed to compute the OPT estimate. In \cite{LJW}, a different way to construct the random
partition is introduced where the size of the partition grows linearly
instead of geometrically as in OPT. This allows the authors to use
sequential importance sampling to sample from the posterior
distribution. This “Bayesian Sequential Partition” (BSP) method is
computationally more scalable to higher dimensions than the OPT
method. As an application, the methods were used to estimate
within-class densities in classification problems, thereby obtaining
approximations to the Bayes classifier. When tested on standard data
sets with $p$ ranging from 10-50, the results are competitive to those
from leading classification methods such as SVM and boosted tree.

The purpose of the current paper is to address the following questions
on such Bayesian density estimates based on partition learning.
Question 1: what is the class of density functions that can be “well
estimated” by these methods. Question 2: what is the rate in which the
posterior distribution is concentrated around the true density as the
sample size increases? For question 1, our analysis will make use of
some results from a companion paper \cite{LW2014} on the properties of
sieve MLEs where the sieve is constructed by considering simple
functions supported by binary partitions of growing
sizes. Specifically, \cite{LW2014} showed that if the true density can
be approximated in Hellinger distance at a rate of $I^{-r}$ where $I$ is the
size of the partition, then the convergence rate of the sieve-MLE
density estimate is $O(n^{-r/(2r+1)} )$ up to $\log n$ terms, where $n$
is the sample size. We note that the term ``well estimated'' in question
1 can now be given a more specific meaning, namely that the
convergence rate of the estimate should not deteriorate fast when the
dimension $p$ of the sample space is large. \cite{LW2014} gave examples
of functions for which approximation rate $I^{-r}$ is not affected by
$p$ much. These include functions satisfying mixed-H\"older continuity
conditions or functions with spatial sparsity as characterized by fast
decay of Haar wavelet coefficients. It is well known that sieve MLEs
are closely related to penalized estimates which is in turn related to
Bayesian methods (\cite{wahba1978}, \cite{shen1997} and \cite{shen2001}). Thus we expect
that the class of density well estimated by the Bayesian methods
should be the same class analyzed by \cite{LW2014}, i.e. the class of
densities that can be approximated at rate $I^{-r}$ for some $r>0$.
We will see that this is indeed true as a consequence of our main
result. Our main result (Theorem 2.1) also provides the answer to the
second question: it shows that the posterior probability is
concentrated in a shrinking Hellinger ball around the true density,
where the radius of the ball is $O(n^{-r/(2r+1)} )$ up to $\log n$
terms. 

Although the convergence rate of the Bayesian method matches that of
the sieve MLE, there is an important difference. While this rate is
achieved by the Bayesian method without requiring any knowledge of the
constant $r$ that characterizes the complexity of the true density
function, the sieve MLE can achieve this same rate only if the size of
the sieve grows at a rate that depends on $r$, specifically, the size of
the partition must be of order $n^{- 1/(2r+1)}$. In other words, the
Bayesian estimate is adaptive to the complexity of the true density
while the sieve MLE is not.  This is an important difference in
practice.

We now briefly review previous literature
on convergence rate of posterior distributions. In breakthrough works \cite{ghosal00}
and \cite{shen2001}, the authors developed general theory on posterior
convergence rates and discussed several applications. Following this theory,
most results have focused on mixture models (\cite{lo1984} and
\cite{ferguson1983}), because these models allow the study of smooth
density functions. Some elegant works include \cite{ghosal2001} and
\cite{ghosal2007}, which studied the concentration rate of the posterior
distribution under Dirichlet mixtures of Gaussian priors, and
\cite{ghosal2001Bern} and \cite{rousseau2010}, which examined the
posterior concentration rate under the mixtures of Beta
priors. Compared to the previous literature, one major improvement of
our result is
that it can deal with multivariate cases. In particular, the rate
attained by our estimate is independent of the dimension $p$, if the
true density falls within the support of the prior. When specialized to
the univariate case, it still coincides with the previous
results. For instance, for one dimensional H\"older space with
parameters between $0$ and $1$, our result is minimax up to a $\log n$ term. Another contribution is
that our result can adapt to the unknown
complexity of the density function.  There has been few adaptive rate
results for Bayesian density estimates in the literature (see
\cite{hoffman2015} for a more extensive review of recent results on
adaptive posterior concentration rates). A notable
exception is in \cite{rousseau2010}, where the
author obtained adaptive posterior concentration rates for one-dimensional H\"older spaces under
mixture Beta priors. Here, our result can adapt to a broader range of
density functions, including spatially sparse density functions, H\"older
continuous functions, and functions of bounded variation. We gain this
advantage at a cost of relatively poor performance for functions with
higher order smoothness. It is our belief that in the multivariate
case, smoothness is not the best condition to characterize functions
that can be well estimates. The rate under usual smoothness condition
is $n^{-(\kappa/(2 \kappa+p))}$ (\cite{stone1980}), where $\kappa$ is the number of
derivatives. Thus high order smoothness cannot guarantee good
convergence when $p$ is large. 

The article is organized as follows. In Section \ref{sec:notation} we
define the prior distribution and summarize our main results on
posterior concentration rate. We express the posterior measure of the
complement of a Hellinger ball as a ratio, where the numerator is the product of prior
probability and the likelihood, and the denominator is the normalizing
factor. In order to derive the concentration rate, we need to upper
bound the numerator and lower bound the dominator. In Section
\ref{sec:upperbd} and Section \ref{sec:lowerbd}, we discuss these upper
and lower bounds respectively. Finally, in Section
\ref{sec:finalproof}, we combine these results to derive the
posterior concentration rate.

%% file: result.tex
\section{Main results on posterior concentration rate}
\label{sec:notation}
In this paper, we focus on the density estimation problem in the
$p$-dimensional Euclidean space. Let $(\Omega, \mathcal{B})$ be a
measurable space and $f_0$ be a compactly supported
density function with respect to the Lebesgue measure $\mu$. $Y_1,
Y_2, \cdots, Y_n$ is a sequence of independent variables distributed
according to $f_0$. After
translation and scaling, we can always assume that the support of
$f_0$ is contained in the unit cube in $\mathbb{R}^p$. Translating
this into notations,
we assume that $\Omega=\{(y^1, y^2, \cdots, y^p): y^l
\in [0,1]\}$. $\mathcal{F}=\{f \mbox{ is a nonnegative measurable
  function}\\ \mbox{on }\Omega: \int_{\Omega}
f d\mu =1 \}$ denotes the collection of all the density functions on 
$(\Omega,
\mathcal{B}, \mu)$. Then $\mathcal{F}$ constitutes the parameter space in
this problem. Note that $\mathcal{F}$ is an infinite dimensional
parameter space. 

\subsection{Densities on binary partitions}
\label{subsec:sieve}
To address the infinite dimensionality of $\mathcal{F}$, we construct a sequence of
finte dimensional approximating spaces $\Theta_1, \Theta_2, \cdots,
\Theta_I, \cdots$ based on \emph{binary partitions}. With growing
complexity, these spaces provides more and more accurate
approximations to the initial parameter space $\mathcal{F}$. Here, we use a recursive procedure to define a binary partition
with $I$ subregions of the unit cube in
$\mathbb{R}^p$. Let $\Omega=\{(y^1, y^2, \cdots, y^p): y^l
\in [0,1]\}$ be the unit cube in $\mathbb{R}^p$. In the first step, we choose
one of the coordinates $y^l$ and cut $\Omega$ into two subregions along the midpoint of
the range of $y^l$. That is, $\Omega=\Omega^l_0 \cup \Omega^l_1$, where $\Omega^l_0 = \{y \in \Omega:
y^l \leq 1/2\}$ and $\Omega^l_1 = \Omega \backslash \Omega^l_0$. In this way, we get a
partition with two subregions. Note that the total number of possible partitions
after the first step is equal to the dimension $p$. Suppose after $I-1$ steps of the
recursion, we have obtained a partition
$\{ \Omega_i \}_{i=1}^I$ with $I$ subregions. In the $I$-th step, further
partitioning of the region is defined as follows: 
\begin{enumerate}
\item Choose a region from $\Omega_1, \cdots, \Omega_I$. Denote it as
  $\Omega_{i_0}$. 
\item Choose one coodinate $y^l$ and divide $\Omega_{i_0}$ into two subregions
  along the midpoint of the range of $y^l$.
\end{enumerate}
Such a partition obtained by $I-1$ recursive steps is called
a binary partition of size $I$. Figure~\ref{fig:binarypartition}
displays all possible two dimensional binary partitions when $I$ is 1, 2 and 3.

\begin{figure} 
\label{fig:binarypartition}
\centering
\includegraphics[scale=0.4]{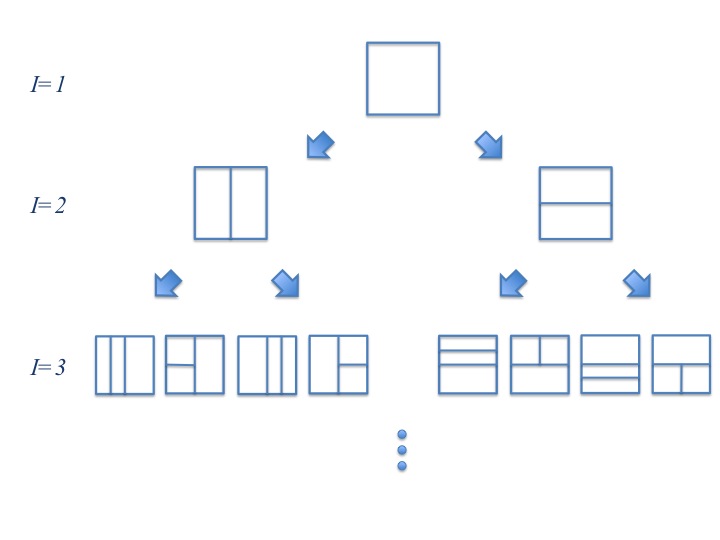}
\caption{Binary partitions}
\end{figure}

Now, let
\begin{eqnarray*}
\Theta_I = \{f \in \Theta: f=\sum_{i=1}^I \beta_i
\mathbbm{1}_{\Omega_i},\sum_{i=1}^I \beta_i \mu(\Omega_i) =1,\\
\{\Omega_i\}_{i=1}^I\ \mbox{is a\ binary\ partition\ of}\ \Omega\ \mbox{of\ size}\ I.\}.
\end{eqnarray*}
Then, $\Theta_I$ is the collection of the density functions
supported by the binary partitions of size $I$. They constitute a sequence of approximating spaces
(i.e. a sieve, see \cite{Grenander} and \cite{shen1994} for background on sieve
theory). Let $\Theta = \cup_{I=1}^\infty \Theta_I$ be the space
containing all the density functions supported by the binary
partitions. Then $\Theta$ is an approximation of the initial parameter
space $\mathcal{F}$ to certain approximation error which will be characterized later. 

We take the metric on $\mathcal{F}$, $\Theta$ and $\Theta_I$ to be Hellinger distance, which is defined to be
\begin{equation}
\label{al:hellinger}
\rho(f,g)=(\int_{\Omega} (\sqrt{f(y)} - \sqrt{g(y)} )^2
dy)^{1/2},\ f,g \in \Theta.
\end{equation}
For $f, g \in \Theta_I$, let $f = \sum_{i=1}^I \beta_i^1
\mathbbm{1}_{\Omega_i^1}$, $g=\sum_{i=1}^I \beta_i^2
\mathbbm{1}_{\Omega_i^2}$, where $\{\Omega_i^1\}_{i=1}^I$ and $\{
\Omega_i^2 \}_{i=1}^I$ are binary partitions of $\Omega$. Then the Hellinger distance
between $f_I^1$ and $f_I^2$ can be written as
\begin{equation}
\rho^2(f, g)=\sum_{i=1}^I \sum_{j=1}^I (\sqrt{\beta_i^1} -
\sqrt{\beta_j^2})^2 \mu(\Omega_i^1 \cap \Omega_j^2).
\end{equation}

We will aso use Kullback-Leibler divergence and the variance of
the log-likelihood ratio based on a single observation $Y$, which are
defined to be 
\begin{equation}
  \label{eq:kl}
  K(f_0, f) = \mathbb{E}_{f_0} \Big( \log \frac{f_0(Y)}{ f(Y)} \Big), 
\end{equation}
and 
\begin{equation}
  \label{eq:varlog}
  V(f_0, f) = \mbox{Var}_{f_0} \Big ( \log \frac{f_0(Y)}{ f(Y)} \Big).
\end{equation}

\subsection{Approximation error}
\label{subsec:approx}
The accuracy of the approximation to the true density by the elements
in $\Theta$ is
formulated in the following way. A density function $f \in \mathcal{F}$ is
said to be well approximated by elements in $\Theta$, if
there exits a sequence of $f_I \in \Theta_I$, satisfying that
$\rho(f_I, f) = O( I^{-r} )(r>0)$. This means that there
exists constant $A_1$ and $A_2$, such that $A_1 I^{-r} \leq \min_{g
  \in \Theta_I} \rho(g, f) \leq \rho(f_I,
f) \leq A_2 I^{-r}$. Let $\mathcal{F}_0$ be the collection of these
density functions. We will first derive posterior concentration rate
for the elements in $\mathcal{F}_0$ in terms of the parameter $r$. For
different function classes, this approximation rate $r$ can be
calculated explicitly. This type of results has been discussed in a
parallel paper (\cite{LW2014}). In addition to this, we also assume that
$f_0$ has finite second moment. 

We want to point out that, based on the minimaxity of the Bayes
estimator, it is necessary to restrict our attention to a subset of $\mathcal{F}$. 
In \cite{farrell1967} and \cite{birge1998}, the authors demonstrated that it
is impossible to find an estimator which works uniformly well for every
$f$ in $\mathcal{F}$. This is the case because for any estimator $\hat{f}$, there always exists
$f \in \mathcal{F}$ for which $\hat{f}$ is inconsistent.

\subsection{Prior specification}
\label{subsec:prior}
An ideal prior $\Pi$ on $\Theta= \cup_{I=1}^\infty \Theta_I$ is supposed to be capable of
balancing the approximation error and the complexity of $\Theta$. The
prior in this paper penalizes the size of the partition in the sense
that the probability mass on each $\Theta_I$ is proportional to $\exp(- \lambda I \log I)$. Given a
sample of size $n$, we restrict our attention to
$\Theta_n =\cup_{I=1}^{n} \Theta_I$, because in practice it is not meaningful to study
a partition with the number of subregions greater than the sample
size. This is to say, when $I\leq n$, $\Pi(\Theta_I) \propto \exp(
-\lambda I \log I)$, otherwise $\Pi(\Theta_I)=0$. 

If we use $T_I$ to denote the total number of possible
partitions of size $I$, then it is not hard to see that $\log T_I \leq
c^* I \log I$, where $c^*$ is a constant. Within each $\Theta_I$, the
prior is uniform across all binary partitions. In other words, let
$\{\Omega_i\}_{i=1}^I$ be a binary partition of $\Omega$ of size $I$,
and $\mathcal{F}(\{\Omega_i\}_{i=1}^I )$ is the collection of
piecewise constant density functions on this partition
(i.e. $\mathcal{F}(\{\Omega_i\}_{i=1}^I ) = \{ f = \sum_{i=1}^I
\frac{\theta_i} {|\Omega_i| }
 \mathbbm{1}_{\Omega_i}: \sum_{i=1}^I \theta_i =1 \mbox{ and } \theta_i \geq 0,
 i=1, \ldots, I\}$),
then  
\begin{equation}
  \label{eq:prioronpartition}
\Pi \left( \mathcal{F} \left( \{\Omega_i \}_{i=1}^I \right) \right)
\propto \exp(-\lambda I \log I) / T_I.
\end{equation}

Given a partition $\{\Omega_i\}_{i=1}^I$, the weights $\theta_i $ on
the subregions follow
a truncated Dirichlet distribution with parameters all equal to
$\alpha$ ($\alpha<1$). This is to say, for $ x_1, \cdots, x_I
> \tau$ and $\sum_{i=1}^I x_i =1$,  
\begin{eqnarray}
  \label{eq:prioronweight}
\nonumber  && \Pi\left(f = \sum_{i=1}^{I} \frac{\theta_i}{|\Omega_i|} \mathbbm{1}_{\Omega_i}:
\theta_1 \in dx_1,
  \cdots, \theta_I \in dx_I | f \in \mathcal{F}\left( \{\Omega_i\}_{i=1}^I \right) \right) \\
&\propto& \frac{\Gamma(\alpha I)}{(\Gamma(\alpha))^I} \prod_{i=1}^I
x_i ^{\alpha -1},
\end{eqnarray}
otherwise, the prior probability is zero. $\tau$ is the truncation
parameter. In this paper, we set $\tau$ to be $D I^{-\kappa}$ ($D, \kappa>0$).

\subsection{Posterior concentration rate}
We are interested in how fast the posterior probability measure concentrates
around the true the density $f_0$. Under the prior specified above, the posterior probability is the random measure
given by 
\begin{equation*}
  \Pi (B | Y_1, \cdots, Y_n) = \frac{\int_B \prod_{j=1}^{n} f(Y_j)
    d\Pi (f)}{ \int_{\Theta} \prod_{j=1}^n f(Y_j) d \Pi(f)}.
\end{equation*}
A Bayesian estimator is said to be \emph{consistent} if the posterior
distribution concentrates on arbitrarily small neighborhoods of
$f_0$, with probability tending to 1 under $P_0^n$ ($P_0$ is the
probability measure corresponding to the density function $f_0$). The
posterior concentration rate refers to the
rate at which these neighborhoods shrink to zero while still
possessing most of the posterior mass. More explicitly, we want to
find a sequence $\epsilon_n \rightarrow 0$, such that for sufficiently
large $M$, 
\begin{equation*}
  \Pi(f: \rho(f, f_0) \geq M \epsilon_n |Y_1, \cdots, Y_n) \rightarrow
  0\ \mbox{in}\ P_0^n-\mbox{probability}.
\end{equation*}

The following theorem gives the posterior concentration rate under the
prior probability specified in Section \ref{subsec:prior}.
\begin{theorem}
\label{thm:concentrationrate}
  $Y_1, \cdots, Y_n$ is a sequence of independent random variables
  distributed according to $f_0$. $P_0$ is the probability measure
  corresponding to $f_0$. $\Theta$ is the collection of all the
  $p$-dimensional density functions supported by the binary partitions
  as defined in Section \ref{subsec:sieve}. The prior
  distribution on $\Theta$ is as specified in Section
  \ref{subsec:prior}. If $f_0 \in \mathcal{F}_0$ and $\kappa > \max(2,
  4r)$, then $\epsilon_n = n^{-\frac{r}{2r+1}} (\log n)^{2+
    \frac{1}{2r}}$ is posterior concentration rate.
\end{theorem}

The strategy to show this theorem is to write the posterior
probability measure as
\begin{eqnarray}
  \label{eq:posteriorprob}
 \nonumber && \Pi( f: \rho(f, f_0) \geq M \epsilon_n|Y_1, \cdots, Y_n)
 \\ 
&=& \frac{\sum_{I=1}^ {\infty} \int_{ \{f: \rho(f, f_0) \geq M \epsilon_n \} \cap
     \Theta_I}  \prod_{j=1}^{n} \frac{f(Y_j)}{f_0(Y_j)}
    d\Pi (f)} {\sum_{I=1}^ {\infty} \int_{\Theta_I} \prod_{j=1}^n
   \frac{ f(Y_j)} {f_0(Y_j)} d \Pi(f)}.
\end{eqnarray}
The proof still relies on the mechanism developed in the landmark
works \cite{ghosal00}
and \cite{shen2001}. We first derive the upper bounds for the items in the numerator by
employing previous results from the study of empirical process in Section
\ref{sec:upperbd}. Then we lower
bound the prior mass of the shrinking ball around the true density in Section
\ref{sec:lowerbd}. In Section
\ref{sec:finalproof}, these bounds are integrated together, leading to
a complete proof of the posterior concentration rate. 

\subsection{Discussion}
\subsubsection{Comparison to the sieve MLE}
In the companion work \cite{LW2014}, we studied convergence rate of the sieve maximum
likelihood estimators. In that paper, the approximating spaces $\Theta_I$ are
defined in the same way, and we consider the same subset of density
functions $\mathcal{F}_0$. 

For any $f \in \Theta_I$, the log-likelihood is defined to be
\begin{equation*}
  L_n(f) = \sum_{j=1}^n \log f(Y_j) = 
  \sum_{i=1}^I N_i \log \beta_i,
\end{equation*}
where $N_i$ is the count of data points in $\Omega_i$, i.e., $N_i= \mbox{card}\{j: Y_j \in \Omega_i, 1\leq j\leq n\}$. The maximum likelihood
estimator on $\Theta_I$ is defined to
be
\begin{equation*}
  \hat{f}_{n,I} = \arg\max_{f \in \Theta_I} L_n(f).
\end{equation*} 
Next theorem presented the result on convergence rate of sieve MLE. It is
cited from \cite{LW2014}.
\begin{theorem}
\label{th:convergencerate}
For any $f_0 \in \mathcal{F}_0$, $\hat{f}_{n,I}$ is the corresponding maximum
likelihood estimator over $\Theta_I$. $r$ is the
parameter that characterizes the decay rate of the approximation error to $f_0$ by the
elements in $\Theta_I$. Assume that
$n$ and $I$ satisfy
\begin{equation}
  \label{eq:thmC2}
  I = \left( (2^8 A_2^2 r /c_1) \frac{n}{\log n} \right)^{\frac{1}{2r+1}},
\end{equation}
where the constant $c_1$ can be chosen to be in $(0, 1)$, and $A_2$ is
a constant associated with the decay rate of the approximation error. Then the convergence rate of the
sieve MLE is $ n^{-\frac{r}{2r+1}} (\log n)
^{(\frac{1}{2}+\frac{r}{2r+1})}$.  
\end{theorem}

Comparing these two rates, we can easily see that they are of the same order up to a
logarithmic term. However, for the sieve method, in order to
achieve the optimal convergence rate we need to match the size of the
partition $I$ to the sample size $n$. And this matching depends on
some unknown property of the true density function, i.e., the decay
rate of the approximation error $r$. This implies that, in practice it is
computationally infeasible to achieve optimal rate under the frequency setting. On the other hand, under Bayesian settings, by
imposing a prior on $\Theta_I$, we are able
to achieve the optimal rate without any a priori information. This is
one of the major improvements of the Bayesian method. 

\subsubsection{Computational issues}
The total number of binary partitions grows
exponentially in $I$, thus it is urgent to solve the computational
issues. In \cite{LJW}, as we mentioned
before, the authors imposed a very similar prior distribution. By
employing sequential importance sampling, they have designed efficient
algorithm to sample from the posterior distribution. Currently, the
dimension of the problem can be moderately large, saying around 50. 
\subsubsection{Applications to different function classes}
In the parallel paper, we studied decay rates of the approximation error for different
density functions classes, including the densities satisfying a type of
sparsity, the space of bounded variation, and mixed-H\"older
continuous functions. Since in this paper we use the same
approximating spaces, those results still hold. Given this, we can also calculate
the corresponding rates of posterior contraction. Based on the
minimaxity of Bayesian estimator, these rates are at least upper bounds of
minimax convergence rates. In fact, for the one dimensional
density functions of bounded variation, the posterior contraction rate
is $n^{-1/3} (\log n ) ^{5/2}$. If we estimate the
density by wavelet thresholding, the convergence rate is $n^{-1/3}
(\log n)^{1/3}$. As a benchmark, the minimax rate of convergence is $n^{-1/3}$. 

\subsubsection{Univariate case}
In \cite{rousseau2010}, the author investigated rates of convergence
for the posterior distribution under the mixture of Beta prior. The
true density function is assumed to be H\"older continuos on
$[0,1]$. More rigorously, the class of H\"older functions $\mathcal{H} (L, \beta)$ with
regularity function $\beta$ is defined as
the following: let $\kappa$ be the largest integer smaller than
$\beta$, and denote by $f^{(\kappa)}$ its $\kappa$th derivative.
\begin{equation*}
  \mathcal{H} (L, \beta) = \{ f: [0,1] \rightarrow \mathbb{R}:
    |f^{(\kappa)} (x) - f^{(\kappa)} (y) | \leq L |x-y|^{\beta-\kappa} \}.
\end{equation*}
Then, under a class of location mixtures of Beta models, the
concentration rate of the posterior distribution is $n^{-\beta
  /(2\beta +1)}$, up to a $\log n$ term. It is known that the rate
$(n/ \log n) ^{-\beta/ (2\beta+1)}$ is the minimax rate of convergence for class
$\mathcal{H}(\beta,L)$.

Under the prior distribution specified in this paper, we can also
study the posterior contraction rate for the H\"older class. However,
given the piecewise constant approximations, we will only study the
H\"older continuous function on $[0,1]$ with regularity parameter $\beta$ in
$(0,1]$. For this class of density functions, we already calculated the decay
rate of the approximation error in \cite{LW2014}. Then the convergence
rate of the posterior distribution is $n^{-\frac{\beta}{ 2\beta +1}} (\log
n) ^{2+ \frac{1}{2\beta}}$. Up to a $\log n$ term,
this method still achieves the minimax rate of convergence. 

%% file: upperbound.tex
\section{Upper bound of the numerator}
\label{sec:upperbd}
Briefly speaking, the numerator can be bounded by controlling the
complexity of the parameter space $\Theta$. Here, the complexity of
the model is measured by the \emph{metric entropy}. A general discussion of metric entropy can be found in \cite{kolmogorov1992selected}.
In this section, we introduce a form of metric
entropy with bracketing corresponding to the relavent parameter space, 
and provide an upper bound for the metric entropy of the
approximating spaces defined in Section \ref{subsec:sieve}. These
bounds lead to upper bounds for the items in the numerator of (\ref{eq:posteriorprob}).

\begin{definition}
  Let $(\Theta, \rho)$ be a seperable pseudo-metric
  space. $\Theta(\epsilon)$ is a finite set of pairs of functions $\{
  (f_j^L, f_j^U), j=1,\cdots, N\}$ satisfying
  \begin{equation}
    \label{eq:entropyDefC1}
    \rho( f_j^L, f_j^U) \leq \epsilon\ for\ j=1,\cdots,N,
  \end{equation}
and for any $f \in \Theta$, there is a $j$ such that
\begin{equation}
  \label{eq:entropyDefC2}
  f_j^L \leq f \leq f_j^U.
\end{equation}
Let 
\begin{equation}
  \label{eq:entropyDefN}
  N(\epsilon, \Theta, \rho) = min \{card\ \Theta(\epsilon): (\ref{eq:entropyDefC1})\ and\ (\ref{eq:entropyDefC2})\ are\ satisfied
\}.
\end{equation}
Then, we define the metric entropy with bracketing of $\Theta$ to be 
\begin{equation}
  \label{eq:entropyDef}
  H(\epsilon, \Theta, \rho) = \log N(\epsilon, \Theta, \rho).
\end{equation}
\end{definition}

Recall that $\Theta_1, \cdots, \Theta_I, \cdots$ are the approximating spaces defined in section
\ref{subsec:sieve}. The
next lemma is devoted to an upper bound for the bracketing metric entropy of
$\Theta_I$.

\begin{lemma}
  \label{lemma:dballupperbound}
Take $\rho$ to be the Hellinger distance. Let $\Theta_I^d = \{ f \in \Theta_I: \rho(f, f_0) \leq
d\}$. Then, 
\begin{eqnarray}
  \label{eq:dballupperbound}
\nonumber  & &H(\epsilon, \Theta_I^d, \rho) \\
&\leq& I \log p + (I+1) \log(I+1) + \frac{I}{2} \log I + I \log \frac{d}{\epsilon} +c',
\end{eqnarray}
where c is a constant not dependent on I or d. 
\end{lemma}
\begin{proof}
See \cite{LW2014} proof of Lemma 3.1 and Lemma 3.2.
\end{proof}

Our next theorem, which is Theorem 1 in \cite{wong1995}, gives a
uniform exponential bound for likelihood ratios. 

\begin{theorem}[Wong and Shen (1995)]
\label{thm:likelihoodratio}
  There exist positive constants $a>0$, $c$, $c_1$ and $c_2$, such that,
  for any $\epsilon>0$, if 
  \begin{equation}
    \label{eq:entropycondition}
    \int_{\epsilon^2/8}^{\sqrt{2}\epsilon}H^{1/2} ( u/a, \mathcal{P},
    \rho) du \leq c n^{1/2} \epsilon^2,
  \end{equation}
then
\begin{equation*}
  \mathbb{P}_{f_0} \Big( \sup_{\{ \rho(f, f_0) \geq \epsilon, f \in
    \mathcal{P} \}} \prod_{i=1}^n \frac{f(Y_i)}{f_0(Y_i)} \geq \exp( -c_1
  n \epsilon^2) \Big) \leq 4 \exp(-c_2 n \epsilon^2),
\end{equation*}
where $\mathbb{P}_{f_0}$ is understood to be the outer probability
mesure under $f_0$. The constants $c_1$ and $c_2$ can be chosen in
$(0,1)$ and $c$ can be set as $(2/3)^{5/2} /512$.
\end{theorem}

Finally, the next lemma provides an upper bound for the items in the numerator in
(\ref{eq:posteriorprob}) when $I$ is sufficiently large. 

\begin{lemma}
\label{lemma:upperbd}
  Let $\delta_{n,I} = (\frac{I \log I}{n / \log n}) ^{1/2}$. When $n$ and $I$ are sufficiently
  large, we have 
  \begin{equation*}
    \mathbb{P}_{f_0} \Big( \sup_{\{ \rho(f, f_0) \geq \delta_{n,I}, f \in
   \Theta_I \}} \prod_{i=1}^n \frac{f(Y_i)}{f_0(Y_i)} \geq \exp( -c_1
  n \delta_{n,I}^2) \Big) \leq 4 \exp(-c_2 n \delta_{n,I}^2).
  \end{equation*}

\end{lemma}

\begin{proof}
 See \cite{LW2014} proof of Corollary 3.1.
\end{proof}
\begin{remark}
 Since the metric entropy decreases as $\epsilon$
increases, this lemma also holds for any $\epsilon \geq
\delta_{n,I}$. This property is quite useful in the proof of the main
theorem.
\end{remark}

%% file: lowerbound.tex
\section{Lower bound of the denominator}
\label{sec:lowerbd}
In this section, we study how the prior distribution
concentrates on the shrinking neighborhoods around the true density
function. This is the key to bounding the denominator of
(\ref{eq:posteriorprob}) from below. We develop our results
through a series of lemmas. The connection between the lower bounds of the items
in the denominator of (\ref{eq:posteriorprob}) and the concentration rate
of the prior distribution is first derived (\ref{lemma:shen2001}). By employing a property of Dirichlet
distribution (Lemma {\ref{lemma:dirichlet}) and inequalities
bounding Kullback-Leibler divergence by Hellinger distance (Lemma
  \ref{lemma:hellingertokl}), we obtain lower bounds of the terms
  in the denominator of (\ref{eq:posteriorprob}) in Lemma
  \ref{lemma:lowerbd}.

To begin with, we cite a result from \cite{shen2001}. In this lemma,
it is shown that with probability close to 1, the
denominator is bounded from below by the prior probability mass
concentrating on a ball around $f_0$ multiplied by a coefficient
depending on the radius of the ball. 
\begin{lemma}[Shen and Wasserman (2001) Lemma 1] 
\label{lemma:shen2001}
Let $K(\cdot, \cdot)$ and $V(\cdot, \cdot)$ be as defined in
(\ref{eq:kl}) and (\ref{eq:varlog}), and let $S(t) = \{ f \in \Omega : K(f_0, f) \leq t, V(f_0, f) \leq t \}
$. Set $S_n =S(t_n)$. When $t_n$ is a sequence of positive numbers satisfying $nt_n
\rightarrow \infty$,
\begin{equation*}
  \mathbb{P}_{f_0}^n \left( \int_{\Omega} \prod_{j=1}^n \frac{f
    (Y_i)}{f_0(Y_i)} d \Pi(f) \leq \frac{1}{2} \Pi (S_n) e^{-2nt_n}
  \right) \leq \frac{2}{nt_n}.
\end{equation*}
\end{lemma}
More explicitly, from this lemma we learn that, given the condition $n
t_n \rightarrow \infty$, $\int_{\Omega} \prod_{j=1}^n \frac{f
    (Y_i)}{f_0(Y_i)} d \Pi(f) \geq \frac{1}{2} \Pi (S_n) e^{-2nt_n}$
  with probability close to 1.

It is well known that Hellinger distance can be bounded by the
Kullback-Leibler divergence. In \cite{wong1995}, they showed that
the other direction also holds under an integrability condition.
Their results are summarized in the lemma below.

\begin{lemma} [Wong and Shen (1995) Theorem 5]
  \label{lemma:hellingertokl}
  Let $f$, $f_0$ be two densities, $\rho^2(f, f_0) \leq
  \epsilon^2$. Suppose that $M_{\delta} ^2 = \int_{\{f_0/f \geq
    e^{1/\delta} \} } f_0 (f_0 /f)^{\delta} < \infty$ for some $\delta
  \in (0, 1]$. Then for all $\epsilon^2 \leq \frac{1}{2} (
  1-e^{-1})^2$, we have 
  \begin{eqnarray*}
    \int f_0 \log (\frac{f_0}{f}) \leq \big[ 6 + \frac{2 \log
      2}{(1-e^{-1})^2} + \frac{8}{\delta} \max \big( 1 , \log
    (\frac{M_{\delta}}{\epsilon}) \big) \big] \epsilon^2,\\
\int f_0 \big( \log (\frac{f_0}{f}) \big)^2 \leq 5\epsilon^2 \big[
\frac{1}{\delta} \max \big( 1, \log (\frac{M_{\delta}}{\epsilon} )
\big) \big]^2.
  \end{eqnarray*}
\end{lemma}
From the proceeding lemma, we learn that, if $\rho^2(f, f_0) \leq
\epsilon^2$, then 
\begin{equation*}
  \max \Big( K(f_0, f), \mathbb{E}_{f_0}\big( (\log \frac{f(Y)}{f_0(Y)}
  )^2 \big) \Big) = O \big( \epsilon^2 (\log
  \frac{M_{\delta}}{\epsilon} )^2 \big).
\end{equation*}
This further implies that, there exists a constant $L$, such
that
\begin{eqnarray}
  \label{eq:hellingertokl}
  \nonumber && \Big \{ f: \rho(f, f_0) \leq \frac{L \epsilon} { (\log
  \frac{M_{\delta}}{\epsilon})^2 } \Big\} \\
& \subset& \Big \{ f: K(f_0, f) \leq  \epsilon^2, \mathbb{E}_{f_0}
\big( (\log  \frac{f(Y)}{f_0(Y)} ) ^2\big) \leq  \epsilon^2 \Big \}.
\end{eqnarray}

This lemma allows us to work on a Hellinger ball instead of a
Kullback-Leibler one. The transition is necessary because it is more
straightforward to apply a property of the Dirichlet distribution to
estimate the probability mass on a Hellinger ball around the true density
function.  In the lemma below, this particular property of the
Dirichlet distribution is stated in terms of $L_1$ distance, which is
equivalent to the Hellinger distance. We want to point out that this lemma is a variation of Lemma 6.1 in \cite{ghosal00} and the proof is adapted from their
paper.

\begin{lemma}
\label{lemma:dirichlet}
  Let $(X_1, \cdots, X_I)$ be distributed according to the truncated
  Dirichlet distribution (\ref{eq:prioronweight}) with truncation parameter $\tau$. Let $(x_{10}, \cdots,
  x_{I0})$ be any point on the $I$-simplex. Let $\epsilon <
  1/I$. Assume that $\tau < \epsilon^2$. Then 
  \begin{equation}
\label{eq:dir}
    P \Big( \sum_{i=1}^I |X_i -x_{i0} | \leq 2\epsilon \Big) \geq
    \frac{\Gamma(\alpha I)}{ (\Gamma(\alpha))^I} (\epsilon^2 - \tau)^I.
  \end{equation}
\end{lemma}
\begin{proof}
  We can find an index $i$ such that $x_{i0} > 1/I$. By relabeling, we
  can assume that $i=I$. if $|x_i -x_{i0} | \leq \epsilon^2$ for $i =
  1, \cdots, I-1$, then 
  \begin{equation*}
    \sum_{i=1}^{I-1} x_i \leq 1 - x_{I0} + (I-1) \epsilon^2 \leq
    (I-1)(\epsilon^2 + 1/I) \leq 1- \epsilon^2 <1.
  \end{equation*}

Therefore, there exists $x=(x_1, \cdots, x_I)$ in the simplex with
these first $I-1$ coordinates. And
\begin{equation*}
  \sum_{i=1}^I |x_i - x_{i0}| \leq 2 \sum_{i=1}^{I-1} |x_i -x_{i0}|
  \leq 2\epsilon^2 (I-1) \leq 2\epsilon.
\end{equation*}
Therefore, the probability on the left hand side of (\ref{eq:dir}) is
bounded below by
\begin{eqnarray*}
  && P( |X_i -x_{i0}| \leq \epsilon^2, i =1, \cdots, I-1)\\
&\geq& \frac{\Gamma(\alpha I)}{ (\Gamma(\alpha))^I} \prod_{i=1}^{I-1}
\int_{max((x_{i0}- \epsilon^2), \tau)} ^{min((x_{i0}+\epsilon^2) , 1)}
x_i^{\alpha-1} dx_i (1-\tau).
\end{eqnarray*}
Since $\alpha <1$, we can lower bound the integrand by $1$ and the
interval of integration contains at least an interval of length
$\epsilon^2-\tau$. Therefore, the result above can be further lower
bounded by 
\begin{equation*}
  \frac{\Gamma(\alpha I)}{ (\Gamma(\alpha))^I} (\epsilon^2
  -\tau)^{I-1} (1-\tau) \geq \frac{\Gamma(\alpha I)}{ (\Gamma(\alpha))^I} (\epsilon^2
  -\tau)^I.
\end{equation*}
This finishes the proof. 
\end{proof}

Now, we are ready to derive lower bounds for the prior probability
mass on $\Theta_I$'s when $I$ varies within a certain range. Before
stating the result, we want to briefly review the assumptions we made
in Section \ref{subsec:approx} and Section
\ref{subsec:prior}. First, in terms of approximation error, we
assume that for any $f_0 \in \mathcal{F}_0$, there exists a
sequence of $f_I  \in \Theta_I$, such that $A_1 I^{-r} \leq \min_{g
  \in \Theta_I} \rho(g, f) \leq \rho(f_I,
f) \leq A_2 I^{-r}$ for some positive constants $A_1$ and $A_2$. Second, we imposed a moment condition on
$\mathcal{F}_0$. For any $f \in \mathcal{F}_0$, we assume that $\int f^{2} <
\infty$. At last, given a partition of size $I$, the weights on the
subregions within the partition follow a Dirichlet
distribution truncated from below, with the truncation parameter $\tau = D
I^{-\kappa}$ ($D, \kappa >0$). Under these three assumptions, we will derive the lower bound
in the lemma below.

\begin{lemma}
  \label{lemma:lowerbd}
Assume that $f_0 \in \mathcal{F}_0$. $\Pi$ is the prior probability
specified in Section \ref{subsec:prior}, with $\kappa > \max( 2, 4r)$. Let $t_{n,I} =
\epsilon_{n,I}^2= \frac{I \log I}
{n/ \log n}$. When
$I = n^{\frac{1}{2r+1}}$, we have
\begin{eqnarray*}
 &&\mathbb{P}_{f_0}^n \Big( \int_{\Theta_I} \prod_{j=1}^n \frac{f
    (Y_i)}{f_0(Y_i)} d \Pi(f) \\
&& \ \ \ \ \ \ \ \leq \frac{1}{2} \Pi (\Theta_I)
  \exp (-2nt_{n,I}-c^* I \log I -4\omega I \log n -I \log \Gamma(\alpha))
  \Big) \\
&\leq& \frac{2}{nt_{n,I}},
\end{eqnarray*}
where $\omega= \max(1, 1/2r)$.
\end{lemma}
\begin{proof}
Let $S_{n,I} = \{ f \in
\Theta_I : K(f_0, f) \leq t_{n,I}, V(f_0, f) \leq t_{n,I} \}$. From lemma \ref{lemma:shen2001}, we have the bound 
\begin{equation}
  \label{eq:lowerbd1}
  \mathbb{P}_{f_0}^n \Big( \int_{\Theta_I} \prod_{j=1}^n \frac{f
    (Y_i)}{f_0(Y_i)} d \Pi(f) \leq \frac{1}{2} \Pi (S_{n,I})
  e^{-2nt_{n,I}} \Big) \leq \frac{2}{nt_{n,I}}.
\end{equation}
Next step, we will search a lower bound for $\Pi(S_{n,I})$. The way to
approach this is to find a subset of $S_{n,I}$ to which we can apply
Lemma \ref{lemma:dirichlet}. Our argument is as the following.

  Define $\tilde{S}_{n,I} = \{ f \in \Theta_I : K(f_0, f) \leq
  t_{n,I}, \mathbb{E}_{f_0}\big( ( \log \frac{f_0(Y)}{f(Y)})^2 \big) \leq t_{n,I}
  \}$. Note that $\mathbb{E}_{f_0}\big( ( \log \frac{f_0(Y)}{f(Y)})^2 \big) \geq V(f_0, f)$, we have $\tilde{S}_{n,I} \subset S_{n,I}$. 
From (\ref{eq:hellingertokl}), we know that 
\begin{equation*}
  W_{n,I} := \{ f \in \Theta_I: \rho(f_0, f) \leq \frac{ L
    \epsilon_{n,I} }{ \log \frac{M_{\delta}} {\epsilon_{n,I}} }  \} \subset \tilde{S}_{n,I}.
\end{equation*}
With the truncation parameter $\tau= D I^{-\kappa}$, $M_{\delta} =
O ( I^{\delta \kappa} \int f_0 ^{(1+\delta)}
)$. Furthermore, 
\begin{eqnarray}
\label{eq:radius}
 \nonumber \frac{ \epsilon_{n,I} }{ \log \frac{M_{\delta}}
   {\epsilon_{n,I}} } &=& O \left( \frac{ \left( \frac{I \log I}{n / \log n}
    \right) ^ {1/2} }{  \log \left(I^{\delta \kappa} \int f_0
    ^{(1+\delta)}  (\frac{n / \log n}{I \log I})^{1/2} \right)  }
  \right)\\
&=& O \left( \left(\frac{I \log I}{ n \log n }\right)^{1/2} \right).
\end{eqnarray}
Under the assumptions that $I= n^{ \frac{1}{1+2r}}$, there exists $f_I \in \Theta_I$, such
that $\rho( f_0, f_I) < \frac{ L
    \epsilon_{n,I} }{ \log \frac{M_{\delta}} {\epsilon_{n,I}} }$. 
If we define
\begin{equation*}
  \tilde{W}_{n,I} := \{ f\in \Theta_I: \rho (f, f_I) \leq  \frac{ L
    \epsilon_{n,I} }{ \log \frac{M_{\delta}} {\epsilon_{n,I}} }- \rho(f_0, f_I) \},
\end{equation*}
by triangle inequality, we know that $\tilde{W}_{n,I} \subset
W_{n,I}$. Together with the previous result, we claim that there exists a constant
$L'$, such that 
  \begin{equation*}
    \tilde{B}_{n,I} := \{ f\in \Theta_I: \rho (f, f_I) \leq L'
    \left(\frac{I \log I }{ n \log n }\right)^{1/2} \} \subset \tilde{W}_{n,I},
  \end{equation*}
Next, from the fact $ \rho^2(f, g) \leq \|f
-g \|_{L_1}$, we have
\begin{equation*}
  B_{n,I} := \{ f\in \Theta_I: \|f_I - f\|_{L_1}  \leq  \frac{
    L'^2 I \log I}{ n \log n } \} \subset \tilde{B}_{n,I}.
\end{equation*}
Note that $\Pi(B_{n,I} ) = \Pi(\Theta_I )
  \Pi(B_{n,I} | \Theta_I)$. Assume that $f_I$ is supported by the binary partition
$\{ \Omega_{i0} \}_{i=1}^I$. Let $F_0 = \{ f \in \Theta_I: f =
  \sum_{i=1}^I \beta_i \mathbbm{1}_{\Omega_{i0}}, \beta_i \geq 0, \sum_{i=1}^I \beta_i =1 \}$ be the collection of all the density
  functions in $\Theta_I$ which are supported by the same binary
  partition as $f_I$. Then 
  \begin{equation}
    \label{eq:lowerbd2}
    \Pi(B_{n,I} | \Theta_I) \geq
   \Pi(B_{n,I} | F_0) \Pi(F_0 | \Theta_I) \geq \exp ( -c^* I \log I)
  \Pi(B_{n,I} | F_0) .
  \end{equation}
Now we apply Lemma \ref{lemma:dirichlet} to bound $\Pi(B_{n,I} | F_0)$
from below. We will works with an $L_1$-ball with radius $( \frac{
    L'^2 I \log I }{ n \log n } )^{\omega}$, where $\omega$ is
  chosen to be $\max(1, 1/2r) $. We can always assume that $L' <1$,
  otherwise we can work with a smaller ball instead. Obviously, this ball
  is contained in $B_{n,I}$. When $I= n^{\frac{1}{2r+1}}$, we
  have $( \frac{
    L'^2 I \log I }{ n \log n} )^{\omega} < \frac{1}{I}$. Under
  the assumptions $\kappa > \max(2, 4r)$, we
know that when $I \geq n^{\gamma_1}$, $ D I^{-\kappa} =
o( (\frac{I \log I}{n \log n })^{2 \omega} )$. By setting
$x_{i0}$ in the lemma to probability mass
  on $\Omega_{i0}$ under $f_I$, we have 
  \begin{eqnarray}
    \label{eq:lowerbd3}
 \nonumber   \Pi(B_{n,I} | F_0 ) &\geq& \frac{\Gamma(\alpha I)
    }{(\Gamma(\alpha))^I} (  (\frac{L'^2 I \log I}{2n \log n })^{2 \omega}  -DI^{-\kappa} )^I \\
&\geq& \exp( -I \log \Gamma(\alpha) -4 \omega I \log n).
  \end{eqnarray}
Combine (\ref{eq:lowerbd1}), (\ref{eq:lowerbd2}) and
(\ref{eq:lowerbd3}) together, we get the desired result.
\end{proof}

%% file: proof.tex
\section{Proof of Theorem \ref{thm:concentrationrate}}
\label{sec:finalproof}
In this section, we will combine the upper bound in Section
\ref{sec:upperbd} and the lower bound in Section \ref{sec:lowerbd}
together to derive the posterior concentration rate.
\begin{proof}[Proof of Theorem \ref{thm:concentrationrate}]
  Let $\epsilon_n = n^{-\frac{r}{2r+1}} (\log n)^{2+ \frac{1}{2r}} $ and
  $\eta_{n,I} =  \left(  \frac{I (\log I)^{1/r+1}}{ n / \log
      n}\right)^{1/2} $. First, we divide the items in (\ref{eq:posteriorprob}) into three
  blocks. We define
  \begin{eqnarray*}
    \Rmnum{1}_{Num} = \sum_{I=1}^ {N_1 -1} \int_{ \{f: \rho(f, f_0) \geq M \epsilon_n \} \cap
     \Theta_I}  \prod_{j=1}^{n} \frac{f(Y_j)}{f_0(Y_j)} d \Pi(f), \\
\Rmnum{2}_{Num} =\sum_{I=N_1}^ {N_2} \int_{ \{f: \rho(f, f_0) \geq M \epsilon_n \} \cap
     \Theta_I}  \prod_{j=1}^{n} \frac{f(Y_j)}{f_0(Y_j)} d \Pi(f), \\
\Rmnum{3}_{Num} =\sum_{I=N_2+1}^ {n} \int_{ \{f: \rho(f, f_0) \geq M \epsilon_n \} \cap
     \Theta_I}  \prod_{j=1}^{n} \frac{f(Y_j)}{f_0(Y_j)} d \Pi(f), 
  \end{eqnarray*}
where $N_1 = n^{\frac{1}{2r+1}} (\log n)^{-\frac{1}{r}}$ and $N_2 =
  D n^{\frac{1}{2r+1}} (\log n)^2$. 

We deal with each block in the numerator separately. Roughly
speaking, when $I$ is small, the approximation error to $f_0$
dominates, and these items can be bounded by the Hellinger distance between $f$ and $f_0$. The items in the middle range can be
bounded by controlling the metric entropy of $\Theta_I$. The items in
the last block are negligible because the prior probability
decays to zero fast. 

We assume that there exists a
sequence of $f_I  \in \Theta_I$, such that $A_1 I^{-r} \leq \min_{g
  \in \Theta_I} \rho(g, f) \leq \rho(f_I,
f) \leq A_2 I^{-r}$ for some positive constants $A_1$ and $A_2$. When $I < N_1$, $A_1 I^{-r}$ is greater than
$\epsilon_n$. We can apply Lemma \ref{lemma:upperbd} by setting
$\delta_{n,I}$ to be $A_1 I^{-r}$. Therefore, as $n \rightarrow \infty$, 
\begin{eqnarray*}
  \Rmnum{1}_{Num} &\leq& \sum_{I=1}^{N_1 -1 } \Pi(\Theta_I) \exp (
  -A_1 n I^{-2r}) \\
&\leq& (\sum_{I=1}^{N_1 -1 } \exp ( -2A_1 n I^{-2r})) ^{1/2}.
\end{eqnarray*}
Now, we will estimate the order of the summation in the last line. In order to simplify
the notation, we will discuss the order of $\sum_{I=1}^{N_1 -1}
\exp(-\frac{2A_1 n}{I^{2r}})$ in detail.

We know that the mass is centered around $I= N_1 -1$. Power series
expansion around that point gives
\begin{equation*}
  \sum_{I=1}^{(1-\epsilon) N_1} \leq (1-\epsilon) N_1 \exp \left(
  - \frac{2A_1 n}{((1-\epsilon) N_1)^{2r}} \right),
\end{equation*}
which is a lower order term compared to the last term in the summation
and thus does not contribute significantly to the summation.
Let $1-\delta = \frac{I}{N_1}$, expand
\begin{equation*}
  (1-\delta)^{-2r} = 1+ 2r\delta + \binom{-2r}{2} \delta^2 + o (\delta^2).
\end{equation*}
\begin{eqnarray*}
  \sum_{I=(1-\epsilon) N_1}^{N_1 -1} \exp ( - \frac{2A_1 n}{I^{2r}})
& \leq & \int_{(1-\epsilon) N_1}^{N_1} \exp ( - \frac{2A_1 n}{x^{2r}}) dx \\
  &\sim& \int_0^\epsilon \exp\left (-2A_1 n^{\frac{1}{2r+1}} (\log n)^2
    (1-\delta)^{-2r} \right)
  N_1  d\delta \\
&\sim& \int_0^\epsilon \exp \left( -2A_1 n^{\frac{1}{2r+1}} (\log n)^2 ( 1+ 2r \delta
+o(\delta)) \right) N_1 d\delta \\
&\sim& \frac{1}{4rA_1 (\log n)^{1/r+2}} \exp \left( -2A_1 n^{\frac{1}{2r+1}}
(\log n)^2 \right).
\end{eqnarray*}
Therefore
\begin{equation}
  \label{eq:rate1}
  \Rmnum{1}_{Num} \leq (\log n)^{-1 -\frac{1}{2r}} \exp(-A_1
  n^{\frac{1}{2r+1}} (\log n)^2).
\end{equation}

From Lemma \ref{lemma:upperbd}, we know that if the result applies for
$\delta_{n,I}$, then it also applies to $M \eta_{n,I} >
\delta_{n,I}$. We have that when $N_1 \leq I \leq N_2$,
\begin{eqnarray}
  \label{eq:rate2}
 \nonumber  \Rmnum{2}_{Num} &\leq& \sum_{I=N_1}^ {N_2} \int_{ \{f: \rho(f, f_0) \geq M \eta_{n,I} \} \cap
     \Theta_I}  \prod_{j=1}^{n} \frac{f(Y_j)}{f_0(Y_j)} d \Pi(f) \\
\nonumber &\leq& \sum_{I= N_1} ^{N_2} \exp(
  -\lambda I \log I ) \exp( -M^2  I (\log I)^{1+\frac{1}{r}} \log n) \\
\nonumber &\leq& \left( \sum_{I=N_1}^{N_2} \exp(-2 \lambda I \log I) \right)^{1/2}
\left( \sum_{I=N_1}^{N_2}\exp \left( -2 M^2 I (\log I) ^{1+
      \frac{1}{r}} \log n \right) \right)^{1/2} \\
\nonumber &\sim& \exp \left(-M^2 n^{\frac{1}{2r+1}} (\log n)^2 \right),
\end{eqnarray}
where the last line is obtained by integration by part.

For $\Rmnum{3}_{Num}$, we have 
\begin{eqnarray}
\label{eq:num3}
 \nonumber \Rmnum{3}_{Num} &\leq& \sum_{I= N_2 +1} ^{n} \int_{\Theta_I}
  \prod_{j=1}^n \frac{f(Y_j)}{f_0(Y_j)} d \Pi (f) \\
&\sim& \exp\left(-n \int f_0 \log (f_0)\right) \sum_{I=N_2+1}^{n}
\int_{\Theta_I} \prod_{j=1}^n f(Y_j) d\Pi(f).
\end{eqnarray}
If we use $x_I$ to represent a partition of size $I$, and
$\mathcal{X}_I$ to denote the collection of all binary partitions of
size $I$, then the integral in (\ref{eq:num3}) can be divided into
the integral over each partition as the following:
\begin{eqnarray*}
  & &\Rmnum{3}_{Num}\\
 &\lesssim& \exp\left(-n \int f_0 \log (f_0)\right)
  \sum_{I=N_2+1}^{n} \sum_{x_I \in \mathcal{X}_I} 
  \int_{\theta_1, \ldots, \theta_I} \prod_{j=1}^n f(Y_j | \theta_1,
  \ldots, \theta_I, x_I) \\
&  &\times \Pi (\theta_1, \ldots, \theta_I | x_I ) \Pi (x_I)
d\theta_1 \ldots d\theta_I\\
&\lesssim& \frac{(\Gamma(\alpha))^I}{\Gamma(\alpha I)} (\frac{1}{4I^2}
- \tau)^{-I} \exp\left(-n \int f_0 \log (f_0)\right)\\
&&  \sum_{I=N_2+1}^{n} \frac{\exp(-\lambda I)}{T_I} 
\sum_{x_I \in \mathcal{X}_I}
  \frac{D(\alpha+n_1, \ldots, \alpha+n_I )}{ D(\alpha,
    \ldots, \alpha)} \prod_{i=1}^I \frac{1}{|\Omega_i|^{n_i}} ,
\end{eqnarray*}
where $\frac{(\Gamma(\alpha))^I}{\Gamma(\alpha I)} (\frac{1}{4I^2}
- \tau)^{-I}$ is an upper bound for the normalizing constant of the
truncated Dirichlet distribution. This inequality can be obtained from Lemma \ref{lemma:dirichlet}, because,
\begin{eqnarray*}
  &&\sum_{I=N_2+1}^{n} \sum_{x_I \in \mathcal{X}_I} 
  \int_{\theta_1, \ldots, \theta_I} \prod_{j=1}^n f(Y_j | \theta_1,
  \ldots, \theta_I, x_I) \Pi (\theta_1, \ldots, \theta_I | x_I ) \Pi (x_I)
d\theta_1 \ldots d\theta_I\\
&\leq& \frac{(\Gamma(\alpha))^I}{\Gamma(\alpha I)} (\frac{1}{4I^2}
- \tau)^{-I} \sum_{I=N_2+1}^{n} \sum_{x_I \in \mathcal{X}_I} \\
&&  \int_{\theta_1, \ldots, \theta_I} \prod_{j=1}^n f(Y_j | \theta_1,
  \ldots, \theta_I, x_I)  \mbox{Dir} (\theta_1, \ldots, \theta_{I-1}; \alpha,
  \ldots, \alpha| x_I ) \Pi (x_I)
d\theta_1 \ldots d\theta_I .
\end{eqnarray*}
Now, we focus on the part inside the summation, and apply Stirling's
approximation to the gamma function,
\begin{eqnarray}
\label{eq:num3inside}
\nonumber & & \frac{D(\alpha+n_1, \ldots, \alpha+n_I )}{ D(\alpha,
    \ldots, \alpha)} \prod_{i=1}^I \frac{1}{|\Omega_i|^{n_i}} \\
\nonumber &=& \exp \bigg( \log \Gamma(\alpha I) - I \log \Gamma(\alpha) +
  \sum_{i=1}^I \log \Gamma(\alpha +n_i) \\
\nonumber & & - \log \Gamma(\alpha I +n) +
  \sum_{i=1}^I n_i \log \frac{1}{|\Omega_i|} \bigg)\\
\nonumber &\lesssim& \exp\bigg(  \alpha I \log (\alpha I) - \alpha I - I \log
  \Gamma(\alpha) -(\alpha I +n) \log (\alpha I +n) +\alpha I +n \\
& & + \sum_{i=1}^I [ (\alpha +n_i) \log (\alpha +n_i) -(\alpha
  +n_i) +n_i \log \frac{1}{|\Omega_i|} ] \bigg) ,
\end{eqnarray}
Let $C(\alpha) = 1/ \Gamma(\alpha) -\alpha$, then 
\begin{eqnarray}
\label{eq:num3inside2}
 \nonumber && (\ref{eq:num3inside} ) \\
 &\lesssim& \exp\left( \alpha I \log \frac{\alpha I}{\alpha I +n} - n
   \log (\alpha I +n) + C(\alpha)I + \sum_{i=1}^I n_i \log
   \frac{n_i}{|\Omega_i|} \right).
\end{eqnarray}
Given a partition $\{\Omega_i\}_{i=1}^I$, define $\mu_i =
\int_{\Omega_i} f_0$, $\hat{\mu}_i = n_i /n$, and $\nu_i = \mu_i
/|\Omega_i|$. Then we
have $\sqrt{n}(\hat{\mu}_i - \mu_i) \rightarrow \mathcal{N} (0, \mu_i
(1-\mu_i))$ in distribution. With this result, 
\begin{eqnarray*}
  & &(\ref{eq:num3inside2} )\\
& =& \exp\left( \alpha I \log \frac{\alpha I}{\alpha I +n} - n
   \log \frac{\alpha I +n}{n} + C(\alpha)I +\sum_{i=1}^I n_i \log
   \frac{\hat{\mu}_i}{\mu_i} + \sum_{i=1}^I n_i \log \nu_i \right) \\
&=& \exp \bigg( \alpha I \log \frac{\alpha I}{\alpha I +n} - n
   \log \frac{\alpha I +n}{n} + C(\alpha)I + n \int f_0 \log (f_0) - n
   K (f_0, \hat{f}_{x_I}) \\
& & + \sum_{i=1}^I n_i \log
   \frac{\hat{\mu}_i}{\mu_i} + n \sum_{i=1}^I \log (\nu_i)
   (\hat{\mu}_i -\mu_i)  \bigg) \\
&\sim& \exp \left( \alpha I \log \frac{\alpha I}{\alpha I +n} - n
   \log \frac{\alpha I +n}{n} + C(\alpha)I + n \int f_0 \log (f_0) - n
   K (f_0, \hat{f}_{x_I}) \right).
\end{eqnarray*}
From this result, we know that no matter $ I
\ll n$ or $I$ is comparable to $n$, the integral over each partition
is bounded given that $\lambda$ is large enough. If we plug in this
result into the summation, we have 
\begin{eqnarray*}
  \Rmnum{3}_{Num} &\lesssim& \sum_{I=N_2}^{n} \exp (-I \log I) \\
&\leq& \exp(-D n^{\frac{1}{2r+1}} (\log n)^2 ).
\end{eqnarray*}
Therefore
\begin{eqnarray*}
& & (\ref{eq:posteriorprob}) \\
&\lesssim& \frac{ (\log n)^{-1 -\frac{1}{2r}} \exp ( -A_1 n^{\frac{1}{2r+1}}
(\log n)^2 ) + \exp (-M^2 n^{\frac{1}{2r+1}} (\log n)^2 + \exp(-D n^{\frac{1}{2r+1}} (\log n)^2 )
) } 
  { \sum_{I=1}^ {\infty} \int_{\Theta_I} \prod_{j=1}^n
   \frac{ f(Y_j)} {f_0(Y_j)} d \Pi(f) } \\
&\leq& \frac{ (\log n)^{-1 -\frac{1}{2r}} \exp ( -A_1 n^{\frac{1}{2r+1}}
(\log n)^2 ) + \exp (-M^2 n^{\frac{1}{2r+1}} (\log n)^2 ) + \exp(-D n^{\frac{1}{2r+1}} (\log n)^2 ) }  { \frac{1}{2} \exp\left( -
\frac{2}{2r+1} n^{\frac{1}{2r+1}} (\log n)^2 - (\frac{c^*}{2r+1} +
4\omega ) n^{\frac{1}{2r+1}} \log n - n^{\frac{1}{2r+1}} (\log
\Gamma(\alpha) +1)\right) },
\end{eqnarray*}
where the last inequality is obtained by applying Lemma \ref{lemma:lowerbd}
to the space $\Theta_I$ with $I= n^{\frac{1}{2r+1}}$. The last line
goes to zero when $A_1$, $M^2$ and $D$ are all greater than $\frac{2}{2r+1}$.

Therefore, we have 
\begin{equation*}
  \Pi \left( f: \rho(f, f_0) \geq M \epsilon_n|Y_1, \cdots, Y_n \right) \leq \exp\left( -b n^{\frac{1}{2r+1}} (\log n)^2 \right),
\end{equation*}
with probability tending to 1, where $b$ is a positive constant. This concludes the proof. 
\end{proof}